\documentclass[12pt, a4paper]{amsart}

\usepackage{graphicx}
\usepackage{wrapfig}
\usepackage{comment} 

\usepackage{enumitem} 
\setlist[itemize]{leftmargin=*,parsep=8pt}
\setlist[enumerate]{labelsep=*, leftmargin=1.9pc}

\usepackage[
unicode=true,
plainpages = false,
pdfpagelabels,
bookmarks=true,
bookmarksnumbered=true,
bookmarksopen=true,
breaklinks=true,
backref=false,
colorlinks=true,
linkcolor = DarkBlue,
urlcolor  = DarkBlue,
citecolor = DarkRed,
anchorcolor = green,
hyperindex = true,
hyperfigures
]{hyperref}
\hypersetup{
pdftitle={The Higman operations and  embeddings of recursive groups},
pdfauthor={V.H. Mikaelian},
pdfsubject={Group theory, algebra}
}

\usepackage[dvipsnames]{xcolor}
\colorlet{LightRubineRed}{RubineRed!70!}
\colorlet{Mycolor1}{green!10!orange!90!}
\definecolor{DarkRed}{HTML}{cc0000}
\definecolor{ChapterHeadColor}{HTML}{cc0000}
\definecolor{PartHeadColor}{HTML}{cc0000}
\definecolor{DarkBlue}{HTML}{0000cc}
\definecolor{QuoteColor}{HTML}{665665}

\usepackage[top=2.5cm, bottom=2.5cm, outer=2.5cm, inner=2.5cm, heightrounded
]{geometry}

\newcommand{\BB}{\boldsymbol{\mathscr{S}}}
\newcommand{\B}{{\mathcal B}}
\renewcommand{\S}{{\mathcal S}}
\newcommand{\Zz}{{\mathcal Z}}
\newcommand{\E}{{\mathcal E}}

\newcommand{\Z}{{\mathbb Z}}
\newcommand{\Q}{{\mathbb Q}}
\newcommand{\Co}{{\mathbb C}}

\usepackage[bitstream-charter]{mathdesign}

\theoremstyle{plain}
\newtheorem{Theorem}{Theorem}[section]
\newtheorem{Lemma}[Theorem]{Lemma}

\theoremstyle{definition}
  
\theoremstyle{remark}
\newtheorem{Remark}[Theorem]{Remark} 
\newtheorem{Example}[Theorem]{Example} 

\numberwithin{equation}{section}

\usepackage[bitstream-charter]{mathdesign}

\DeclareSymbolFont{cmsymbols}{OMS}{cmsy}{m}{n}
\SetSymbolFont{cmsymbols}{bold}{OMS}{cmsy}{b}{n}
\DeclareSymbolFontAlphabet{\mathcal}{cmsymbols}

\begin{document}
	


\title[The Higman operations]{\large
The Higman operations and  embeddings\\ of recursive groups}

\author{V.\,H. Mikaelian}

\begin{abstract}
In the context of Higman embeddings of recursive groups into finitely presented groups we  suggest an approach, termed the $H$-machine, which for certain wide classes of groups allows constructive Higman embeddings of recursive groups into finitely presented groups.
The approach is based on  Higman operations, and it explicitly constructs some
specific recursively enumerable sets of integer sequences 
arising during the embeddings.
Specific auxiliary operations are introduced to make the work with  Higman operations a simpler and more intuitive procedure.
Also, an automated mechanism of constructive embeddings of countable groups into $2$-generator groups preserving certain ``patterns'' is mentioned.
\end{abstract}

\date{\today}

\maketitle

\setcounter{tocdepth}{2}

\let\oldtocsection=\tocsection
\let\oldtocsubsection=\tocsubsection
\let\oldtocsubsubsection=\tocsubsubsection
\renewcommand{\tocsection}[2]{\hspace{-12pt}\oldtocsection{#1}{#2}}
\renewcommand{\tocsubsection}[2]{\footnotesize \hspace{6pt} \oldtocsubsection{#1}{#2}}
\renewcommand{\tocsubsubsection}[2]{ \hspace{42pt}\oldtocsubsubsection{#1}{#2}}

{\small \tableofcontents}

\section{Introduction}

\subsection{Higman's embedding theorem}
\label{SU Higman's embedding theprem}
In 1961 Higman proved that \textit{a finitely generated group can be embedded in a finitely presented group if and only if it is recursively presented} \cite{Higman Subgroups in fP groups} (see  definitions and notation in~\ref{SU Basic notation and references}).
In his work Higman extensively uses specific recursively enumerable sets of integer sequences which in some sense ``code'' defining relations of groups. Our objective is to suggest an algorithm for construction of such integer sequences for certain types of groups. 
This allows us to list wide classes of groups for which Higman's famous embedding construction can be constructive and effective.

The approach of \cite{Higman Subgroups in fP groups}
will be very briefly outlined in Section~\ref{SE The main steps of the embedding} below. For now let us just mention the main steps of Higman's construction to distinguish those parts to which our new algorithm concerns.
A finitely generated group 
$$G = \langle\, A \mathrel{|} R \rangle = \langle a_1,a_2,\ldots \mathrel{|} r_1, r_2,\ldots \rangle $$
with recursively enumerable relations $r_1, r_2,\ldots$
can be constructively embedded into a 
$2$-generator group
$T=\langle b,c 
\mathrel{|} r'_1, r'_2,\ldots 
\rangle$ 
where the relations $r'_1=r'_1(b,c),\;r'_2=r'_2(b,c),\ldots$ are certain words on letters $b,c$,  and they also are recursively enumerable.
Then for each $r'_s$,\; $s\!=\!1,2,\ldots$, a unique sequence $f_s$ of integers is compiled (see details in~\ref{SU The main steps of the proof on Higman embedding theorem}) so that the set $\{r'_1, r'_2,\ldots \}$ of relations
is ``coded'' by means of a set $\B=\{f_1, f_2,\ldots \}$ of such sequences. Since the transaction from relations set $R$ to sequences set $\B$ is done via a just few constructive steps, the set $\B$ also is recursively enumerable.

The tedious part of \cite{Higman Subgroups in fP groups}
is to show that $\B$ is recursively enumerable if and only if $\B$ can be constructed by some chain of special operators \eqref{EQ Higman operations}.
And parallel to application of those operations a respective \textit{benign subgroup} is being constructed in the free group $F_3=\langle a,b,c\rangle$ of rank three  (see \ref{SU Benign subgroups}). As this process ends up on construction of $\B$, the respective benign subgroup $A_\B$ is obtained inside $F_3$.

In the final short step the benign subgroup $A_\B$ is used to get another benign subgroup in the free group $F_2=\langle b,c\rangle$ of rank two. Then that new  subgroup is used to embed the group $T$ (and hence also $G$) into a finitely presented group via the ``The Higman Rope Trick'' (see 
p. 219 in \cite{Lyndon Schupp},
and \cite{Higman rope trick}).

\subsection{Our  algorithm for construction of $\B$}
\label{SU An algorithm to construct the set}

In~\cite{Higman Subgroups in fP groups} Higman just relies on \textit{theoretical possibility} for construction of $\B$ via operations \eqref{EQ Higman operations}, without any examples of such construction for particular groups. This is understandable as the objective of the fundamental article \cite{Higman Subgroups in fP groups} is much deeper, and for its purposes it is sufficient to know that such a construction of $\B$  is possible, provided that the set of G{\"o}del numbers specifically constructed for the set $R$ is equal to the range of a certain partial recursive function described by Kleene's  characterization (see references in \ref{SU Basic notation and references} below).

However, it is rather strange that after Higman's result there was no attempt to explicitly find constructions of $\B$ by operations \eqref{EQ Higman operations} for particular groups (at least, we were unable to find them in the literature). 
Investigating the topic we noticed that such construction may be a \textit{doable} task for many classes of groups, such as, the free abelian, metabelian, soluble, nilpotent groups, the additive group of  rational numbers $\Q$, the quasicyclic group $\Co_{p^\infty}$, 
divisible abelian groups, etc. (see examples in \ref{SU The structure of sequences}).  

We suggest an $H$-\textit{machine} algorithm with some generic tools that allow to explicitly construct $\B$ by operations \eqref{EQ Higman operations} without any usage of Kleene's  characterization, at all.
In Example~\ref{EX Back to the example for Z^infty} we show how simple it is to apply the algorithm (see Remark~\ref{RE Comparing the very similarly looking sequences}).

\medskip
The advantage of this approach is that construction of the benign subgroup $A_\B$ and, thus, of the explicit embedding of the given $G$ into a finitely presented group becomes a managable procedure. 

To shorten the routine of work with basic Higman operations we introduce a few \textit{auxiliary operations} 
which make the proofs not only shorter but also, we hope, more intuitive to understand (see notation in \ref{SU Extra auxiliary operations}).

\medskip
Another embedding aspect we touch upon is the manner by which the initial group $G$ is constructively embedded into a $2$-generator group $T$, and how the relations of $T$ can be obtained from those of $G$.
In the literature there is no shortage in constructive embeddings of this type (in fact, the original method of \cite{HigmanNeumannNeumann} already  allows that). 
However, for our purposes we need a method which not
only makes deduction of the relations of $T$ from the relations of $G$ a trivial automated  task,
but also \textit{preserves certain ``patterns''} in the relations 
(for illustration of the ``patterns'' see \cite{Embeddings with universal words} and also examples in subsection \ref{SU The structure of sequences} below).

\subsection*{Acknowledgements}
Application of the methods that we present here allows to build a group answering a question of Bridson and de la Harpe on embedding of $\Q$ into a finitely presented groups mentioned in Problem 14.10 (a)  
in the Kourovka notebook \cite{kourovka}.
Recently a direct solution to that problem was found by Belk, Hyde and Matucci in \cite{Belk Hyde Matucci}, see Remark~\ref{RE embedding of rational group by Belk Hyde Matucci} for details.
Another explicit embedding will be given in \cite{An explicit embedding of Q}.

The current work is supported by the 21T-1A213 grant of SCS MES RA. 

It is a pleasure to me to thank the Referee of the Journal of Group Theory for careful work and for very helpful and encouraging remarks.


\section{Definitions, references, preliminary constructions}

\subsection{Basic notation and references}
\label{SU Basic notation and references}
For general group theory information we refer to textbooks \cite{Robinson, Kargapolov Merzljakov, Rotman}.
For background on free constructions, such as, free products, free products with amalgamated subgroups, HNN-extension we refer to \cite{Lyndon Schupp, Bogopolski,Rotman}. See also the recent note \cite{A modified proof}  where we apply some methods related to free constructions. We use them here without restating the notation again.
Information on varieties of groups can be found in Hanna Neumann's monograph~\cite{HannaNeumann}.

We are going to study recursive groups in the language of Higman operations~\eqref{EQ Higman operations}. Recall that a recursive (or recursively presented) group $G$ is that possessing a presentation
$
G=
\langle\, X \mathrel{|} R \,\rangle
=\langle x_1, \ldots, x_n, \ldots \mathrel{|} r_1, r_2, \ldots\, \rangle
$
with finite or countable set of generators $X$, and with a \textit{recursively enumerable} set of defining relations $R$.
That is to say, to each relation $r_i\in R$ one can assign a  \textit{G{\"o}del number} (see Section 2 in 
\cite{Higman Subgroups in fP groups}
or p. 218 in \cite{Lyndon Schupp})  to interpret $R$ via a set of respective G{\"o}del numbers, and then that set turns out to be the range of a partial recursive function. 
By Kleene's  characterization, a
\textit{partial recursive function} is that obtained from the zero function, the
successor function,
and  the 
identity function
using the operations of  composition, 
primitive recursion, and minimization  (see \cite{Davis, Rogers,Boolos Burgess Jeffrey} for details).

Although Higman's theorem is for \textit{finitely} generated recursive groups,
its analog holds for embeddings of \textit{countably} generated recursive groups into 
finitely presented groups. For, 
a countably generated recursive group can first be effectively embedded into a finitely generated 
recursive group 
(see remark proceeding Corollary on p. 456 in \cite{Higman Subgroups in fP groups}). Thus, in embedding procedures we will not take care of the number of generators, as long as the relations are recursively enumerated.

\subsection{Sets of integer-valued functions and sequences of integers}
\label{SU Sets of integer-valued functions}

Denote by $\mathcal E$ the set of all functions $f : \Z \to \Z$ with finite support $\sup (f)=\{i\in \Z \mathrel{|} f(i)\neq 0\}$.
When $m$ is any positive integer such that 
$\sup (f) \subseteq \{0,1,\ldots, m-1\}$, then we can interpret $f$ as a \textit{sequence} $f=(n_0,\ldots,n_{m-1})$ of length $m$, assuming $f(i)=n_i$ for each index $i=0,\ldots,m\!-\!1$. 
The value $f(i)$ is called the $i$'th coordinate of $f$, or the coordinate of $f$ at the index $i$.
Say,
$f=(0,0,7,-8,5,5,5,5)$
means that 
$f(2)=7$, $f(3)=-8$, 
$f(i)=5$ for $i=4,\ldots,7$, and 
$f(i)=0$ for any $i\le 1$ or $i\ge 8$.
Here the initial $0$ is the $0$'th coordinate, and $7$ is the $2$'nd coordinate.

Depending on the situation we may
interpret
the same function by sequences of different length by adding some zeros to it. Say, the above function $f$ can be interpreted as the sequence $f=(0,0,7,-8,5,5,5,5,0,0,0)$ with three ``new'' coordinates $f(8)=f(9)=f(10)=0$.
And the constant zero function $f(i)=0$ may equally well be interpreted as $f=(0)$ or, say, as $f=(0,0,0,0)$.

\subsection{The Higman operations}
\label{SU The Higman operations on subsets}

Start by two specific subsets of $\E$:
$$
\Zz\!=\!\big\{(0)\big\},\quad\quad  
\S\!=\!\big\{(n, n+\!1) \mathrel{|} n\!\in\! \Z\big\}.
$$
We are going to extensively use the following  operations from \cite{Higman Subgroups in fP groups}:
\begin{equation}
\label{EQ Higman operations}
\tag{H}
\iota,\; \upsilon;\quad\quad
\rho,\; \sigma,\; \tau,\; \theta,\; \zeta,\; \pi,\; \omega_m \;\; (\text{for each $m=1,2,\ldots$})
\end{equation}
which we call \textit{Higman operations} on subsets of $\E$. 
The first two operations are binary functions, and for any subsets $\mathcal A, \mathcal B$ of $\E$ they are defined as just the intersection 
$\iota(\mathcal A,\mathcal B)=\mathcal A \cap \mathcal B$ and the union
$\upsilon(\mathcal A,\mathcal B)=\mathcal A \cup \mathcal B$ of those sets.
The notation a little differs from the original notation $\iota \mathcal A \mathcal B$ and $\upsilon \mathcal A \mathcal B$ of \cite{Higman Subgroups in fP groups} which in our case would cause confusion when used in long formulas together with other operations.

\smallskip
The rest of Higman operations  are unary functions defined on any subset $\mathcal A$ of $\E$ as follows:

$\rho(\mathcal A)$ consists of all $f\!\in\E$ for which there is a $g\in \mathcal A$ such that $f(i)=g(-i)$.

$\sigma(\mathcal A)$ consists of all $f\!\in\E$ for which there is a $g\in \mathcal A$ such that $f(i)=g(i-1)$.

$\tau(\mathcal A)$ consists of all $f\!\in\E$ for which there is a $g\in \mathcal A$ such that $f(0)\!=\!g(1)$, $f(1)\!=\!g(0)$ and  $f(i)\!=\!g(i)$ for $i\!\neq\! 0,1$.

$\theta(\mathcal A)$ consists of all $f\!\in\E$ for which there is a $g\in \mathcal A$ such that $f(i)=g(2i)$. 

$\zeta(\mathcal A)$ consists of all $f\!\in\E$ for which there is a $g\in \mathcal A$ such that $f(i)=g(i)$ for  $i\neq 0$. 

$\pi(\mathcal A)$ consists of all $f\!\in\E$ for which there is a $g\in \mathcal A$ such that $f(i)=g(i)$ for  $i\le 0$.

$\zeta$ and $\pi$ are called \textit{liberations} of $\mathcal A$ in the sense that $\zeta$ \textit{liberates} the sequences on $0$: for every $g\in \mathcal A$ it adds to our set $\mathcal A$ all the functions $f$ which accept \textit{any} value at $0$, but which coincide with $g$ elsewhere. 
And $\pi$ \textit{liberates} the sequences on positive integers: 
for every $g\in \mathcal A$ the operation $\pi$
adds to $\mathcal A$ all the functions which accept \textit{any} values at positive indices, but which coincide with $g$ on zero and on all negative indices.  

For a fixed $m=1,2,\ldots$ the set 
$\omega_m(\mathcal A)$ consists of all $f\!\in\E$ for which for every $i\in \Z$ there is a $g=t_i = \big(f(mi),\;f(mi+1),\ldots,f(mi+m-1)\big)\in \mathcal A$. This operation is called \textit{sequence building} as it constructs the functions $f$ by means of some subsequences $g$ of length $m$ chosen from $\mathcal A$.
Since $\sup (f)$ is finite, either $\mathcal A$ contains the zero function, or $\omega_m(\mathcal A)=\emptyset$.

We may agree to apply the unary Higman operations to individual functions also: a set $\mathcal A=\{f\}$ may consist of a single function $f$ only, so notations like $\rho f,\; \sigma f,\; \tau f$, etc., should cause no confusion.

To get familiar with these operations the reader may check examples and basic lemmas in Section 2 of \cite{Higman Subgroups in fP groups} or in Subsection 2.2 of \cite{A modified proof}.

\smallskip
Following Higman \cite{Higman Subgroups in fP groups} we denote by $\BB$ the set of all subsets of $\E$ which can be obtained from $\Zz$ and $\S$ by any series of operations \eqref{EQ Higman operations}.
The elements in $\BB$ play a key role in the study of recursively presented groups. 
One of our main tasks below is going to be the discovery of many 
``natural'' generic types of  subsets of $\E$ inside $\BB$.

\subsection{Extra auxiliary operations}
\label{SU Extra auxiliary operations}

Our proofs will be much simplified by some auxiliary operations each of which is a combination of a few Higman operations on subsets $\mathcal A$ of $\E$.

For a positive integer $i$ naturally denote by $\sigma^{i} \mathcal A = \sigma \cdots \sigma \mathcal A$ the result of application of $\sigma$ for $i$ times.
Set the inverse $\sigma^{-1}=\rho\sigma\rho$ as follows:
$f \!\in\! \sigma^{-1}(\mathcal A)$ when there is a $g\in \mathcal A$ such\! that\! $f(i)=g(i+1)$.
This allows to define the negative powers of $\sigma$.
Seting $\sigma^0 \mathcal A =\mathcal A$ we have the powers $\sigma^i$ for \textit{any} integer $i\in \Z$. Clearly, $\sigma^i$ just ``shifts'' a sequence $g\in \mathcal A$ by $|i|$ steps to the right or to the left depending on the sign of $i$. 

It is easy to verify that $\sigma^i \zeta \sigma^{-i} \mathcal A$ consists of all functions $f\in \mathcal A$ in which the $i$'th coordinate is liberated. For briefness denote $\zeta_i = \sigma^i \zeta \sigma^{-i}$.
Moreover, for a finite subset  $S=\{i_1,\ldots,i_m\}\subseteq \Z$ denote the result of application of 
$\zeta_{i_1}\cdots \zeta_{i_m}$ by 
$\zeta_{i_1,\ldots,i_m}$ or by $\zeta_S$.
That is, $\zeta_S \mathcal A$ is the set of all those functions $f\in \E$ for which there is some $g\in \mathcal A$ such that 
$f(i)=g(i)$ for each $i\notin S$.

Denote by $\pi' \mathcal A = \rho \pi \rho \mathcal A$ the liberation of $\mathcal A$ 
on all \textit{negative} coordinates, i.e., the set of all functions $f\in \E$ for which there is some $g\in \mathcal A$ such that 
$f(i)=g(i)$ for each $i\ge 0$.
Denote by $\pi_i \mathcal A = \sigma^{i}\pi\sigma^{-i} \mathcal A $
the liberation of $\mathcal A$ on all coordinates \textit{after} the $i$'th coordinate and, similarly, denote by $\pi'_i \mathcal A = \sigma^{-i}\pi'\sigma^{i} \mathcal A $
the liberation of $\mathcal A$ on all coordinates \textit{before} the $i$'th coordinate. In this notation the original Higman operation $\pi$ is nothing but $\pi_0$.

For any integers $k<l$ set $s=l-k-1$. It is not hard to verify that the set 
$\tau_{k,l} \mathcal A = \sigma^k (\tau \sigma)^{\,s} \tau (\sigma^{-1} \tau)^{-s} \sigma^{-k} \mathcal A$
consists of all modified functions of $\mathcal A$ with $k$'th and $l$'th coordinates ``swapped''.
More precisely, $\tau_{k,l} \mathcal A$ is the set of all functions $f\in \E$ for which there is some $g\in \mathcal A$ such that 
$f(k)=g(l)$,  
$f(l)=g(k)$ 
and
$f(i)=g(i)$ for each $i\neq k,l$.
In this notation the Higman operation $\tau$ is nothing but $\tau_{\!0,1}$.

Furthermore, since any permutation $\alpha$ of a finite set $S$ has a transposition decomposition $\alpha=(k_1\, l_1)\cdots (k_m\, l_m)$, we may introduce the set 
$\alpha\mathcal A=\tau_{k_1,l_1}\cdots \tau_{k_m,l_m} \!\mathcal A$ which can be obtained from $\mathcal A$ by respective permutation of coordinates for all $g\in \mathcal A$.\;
Clearly, $\alpha\mathcal A$ is the set of all functions $f$ for which there is a $g\in \mathcal A$ such that 
$f(i)=g\big(\alpha^{-1}(i)\big)$ for any $i\in \Z$. 

For any finite set of indices $S=\{i_1,\ldots,i_m\}$ and
for any $\mathcal A \subseteq \E$
define the \textit{extract} 
$\epsilon_{S} \mathcal A = \epsilon_{i_1,\ldots,i_m} \mathcal A $ to be the $m$-tuples set $\big\{\big(g(i_1),\ldots,g(i_m)\big) \mathrel{|} g\in \mathcal A\big\}$. 
This operation can be constructed using \eqref{EQ Higman operations} as follows. 
For clarity assume $i_1 \!<\cdots <i_m$, and denote 
$S'=\{i_1, i_1\!+\!1,\ldots,
i_m\!-\!1, i_m \} \backslash S$ (the set of all integers from $i_1$ to $i_m$ except those in $S$).
The set 
$\mathcal A_1 = \zeta^{\vphantom{1}}_{S'} \pi'_{i_1} \pi^{\vphantom{1}}_{i_m}\mathcal A$
consists of functions from $\mathcal A$ with \textit{all} coordinates outside $S$ liberated.
And the set $\mathcal A_2 =\zeta_{S}\Zz$ consists of all functions which accept any integer values on  $S=\{i_1,\ldots,i_m\}$, and which are zero elsewhere. 
The intersection $\mathcal A_3 = \iota(\mathcal A_1, \mathcal A_2)$ consists of those functions $f$ for which there is some $g\in \mathcal A$ such that $f(i)=g(i)$ when $i\in S$, and $f(i)=0$ elsewhere.
To get $\epsilon_{S} \mathcal A$ from $\mathcal A_3$ it remains to apply the appropriate permutation $\alpha$  that re-distributes the coordinates of $f\in \mathcal A_3$ at the indices $i_1,\ldots,i_m$ on the
$0,1,\ldots,m-1$ (here $\alpha$ is a permutation of the union $\{i_1,\ldots,i_m;\;0,1,\ldots,m-1\}$).

For any subset $\mathcal A \subseteq \E$ denote $\sup(\mathcal A) = \cup\,\{\sup(f) \mathrel{|} f\in \mathcal A\}$.
The point-wise sum $f+g$ of any functions $f,g\in\E$ is defined as $(f+g)(n)=f(n)+g(n)$,\;  $n\in \Z$.
For any subsets $\mathcal A,\; \mathcal B \subseteq \E$ their sum 
$\mathcal A + \mathcal B$ is the set
$\{f+g  \mathrel{|} f\in \mathcal A,\; g\in \mathcal B  \}$. 
The sum of three or more sets is defined in the same manner.
We are going to use this operation for cases when $\sup(\mathcal A)$ and 
$\sup(\mathcal B)$ are disjoint finite sets.

The intersection of any three or more subsets, such as, $\mathcal A,\; \mathcal B ,\; \mathcal C \subseteq \E$ can be expressed by Higman operations as 
$\iota \big(\iota (\mathcal A,\; \mathcal B),\; \mathcal C\big)$. To have shorter notation record this as 
$\iota_3(\mathcal A, \mathcal B , \mathcal C)$. Similarly define  intersections $\iota_n$ and unions $\upsilon_n$.
In \cite{Higman Subgroups in fP groups} Higman denotes the same by $\iota^{\!2} \!\mathcal A \mathcal B  \mathcal C$, but in our case this would create confusion in long formulas with many operations.

Investigating the subsets of $\E$ in $\BB$  we, in addition to standard operations
\eqref{EQ Higman operations},
may often use the introduced auxiliary operations. This will shorten the proofs without changing the actual set $\BB$, for, we above have representation of each auxiliary operation via \eqref{EQ Higman operations}.

\subsection{The Benign subgroups}
\label{SU Benign subgroups}
The concept of \textit{benign subgroups} is the key group-theoretical notion used 
in \cite{Higman Subgroups in fP groups}
to connect the sets in $\BB$ with subgroups in free groups, needed in  construction of embeddings into finitely presented groups. 
A subgroup $H$ in a finitely generated group $G$ is called a benign subgroup in $G$, if $G$ can be embedded in a finitely presented group $K$ with a finitely generated subgroup $L\le K$ such that $G \cap L = H$.

For basic properties and examples of benign subgroups we refer to Section~3 in \cite{Higman Subgroups in fP groups} or to subsections~3.1 and 3.2 in \cite{A modified proof}.

Below we reserve the letters $K,L$ for these specific groups only. In particular, if we have $K,L$ for an ``old'' group, and then we construct a ``new'' group with a respective finitely presented overgroup and its finitely generated subgroup, we may again denote them by the same letters $K$ and $L$. 
Also, if we have two benign subgroups, say, $H_1$ and $H_2$, we will denote the respective groups by $K_1, K_2$ and $L_1, L_2$.
The context will tell us \textit{for which} benign subgroups they are being considered, and no misunderstanding will occur.

\section{The main steps of embeddings of recursive groups}
\label{SE The main steps of the embedding}

\subsection{Embedding with ``universal'' words in a free group of rank $2$}
\label{SU The universal generators}

Any countable group $G$ is embeddable into a $2$-generator group $T$ \cite{HigmanNeumannNeumann}.
%
Higman's embedding construction \cite{Higman Subgroups in fP groups} starts by some \textit{effective} embedding of $G$ into an appropriate $T$.
In the recent note 
\cite{Embeddings with universal words}
we suggested a method of effective embedding of any countable group $G$ into a $2$-generator group $T$ such that the defining relations of $T$ are straightforward to deduce from relations of $G$. In fact, the very first embedding construction \cite{HigmanNeumannNeumann} (based on free constructions) and some other embedding constructions (based on wreath products, group extensions, etc.) already allow finding the relations of $T$. However, we need a method that not only makes deduction of the relations of $T$ from those of $G$ an \textit{automated task}, but also \textit{preserves certain pattern} in them, as we will see a little later, see
Remark~\ref{RE numerous groups}.
%

Let a countable group $G$ be given as 
$G= F/ \bar R= \langle\, A \mathrel{|} R \rangle = \langle a_1, a_2,\ldots \mathrel{|} r_1, r_2,\ldots \rangle $
where
$F$ is a free group on a countable alphabet $A$, 
and where $\bar R=\langle r_1, r_2,\ldots \rangle^F$ is the normal closure of the set of all defining relations 
$r_s(a_{i_{s,1}},\ldots,a_{i_{s,\,k_s}})$,
$s=1,2,\ldots$\,, in $F$.

In the free group 
$F_2=\langle b,c \rangle$ of rank $2$ 
choose the words:
\begin{equation}
\label{EQ formula for a_(x,y)}
a_i(b,c)
= c^{(b c^i)^{\,2}\, b^{-1}} 
\!(c^{-1})^b
=  \;
b c^{-i} b^{-1}\!  c^{-i} b^{-1}   \!  
c
b\, c^i   b  c^i b^{-2} c^{-1} \! b,
\end{equation}
$i=1,2,\ldots$\;
The map $\gamma: a_i \to a_i(b,c)$
defines a correspondence:
$$
r_s(a_{i_{s,1}},\ldots,a_{i_{s,\,k_s}})
\to
r'_s (b,c)\,=\,
r_s\big(a_{i_{s,1}}\!(b,c),\ldots,a_{i_{s,\,k_s}}\!(b,c)\big)
$$
obtained by replacing 
each $a_{i_{s,j}}$ in $r_s$
by the word $a_{i_{s,j}}(b,c)$,\; $j=1,\ldots,k_s$.
In fact $\gamma$ defines an embedding of $G$ into the $2$-generator group
$$
T =\big\langle b,c 
\;\mathrel{|}\;
r'_1 (b,c),\; r'_2 (b,c),\ldots\,
\big\rangle
$$
given by the relations 
$r'_s (b,c)$,\,
$s=1,2,\ldots$\,,\, on letters $b,c$ (see Theorem~1.1 in \cite{Embeddings with universal words}). 
If $R$ is recursively enumerable, then the set $R'$ of all above relations $r'_s (b,c)$ also is recursively enumerable. That is,  
$T$ is a recursive group, in case $G$ is.

And when $G$ is a torsion-free group, then \eqref{EQ formula for a_(x,y)} can be replaced by shorter words
\begin{equation}
\label{EQ slightly shorter word}
\bar a_i (b,c)
=  \;
c^{(b c^i)^{\,2} b^{\!-1}} \!
=  \,
b c^{-i} b^{-1}\!  c^{-i} b^{-1}   \!  
c
b\, c^i   b  c^i b^{-1}\!.
\end{equation}
Inserting these $\bar a_i (b,c)$ in $r_s$ we get  shorter words $r''_s(b,c)$, and then $\gamma: a_i \to \bar a_i(b,c)$ 
defines an embedding of $G$ into the $2$-generator group
$$
T =\big\langle b,c 
\;\mathrel{|}\;
r''_1 (b,c),\; r''_2 (b,c),\ldots\,
\big\rangle
$$
(see Theorem~3.2 in \cite{Embeddings with universal words}).

\begin{Example}
\label{EX embedding of free abellian into 2-generator group}
Let $G = \big\langle a_1, a_2,\ldots \mathrel{|} [a_k,a_l],\; k,l=1,2,\ldots \big\rangle$ be the free abelian group $\Z^{\infty}$ of countable rank with relations 
$r_s=r_{k,\,l}=[a_k,a_l]$. 
Since $G$ is torsion-free, we can use the shorter formula \eqref{EQ slightly shorter word} to 
map each $a_i$ respectively to $\bar a_i (b,c)$ in order to get the embedding of $G$ into the $2$-generator recursive group:
$$
T =\big\langle b,c 
\;\mathrel{|}\;
\big[
c^{(b c^k)^{\,2} b^{\!-1}} 
\!\!\!\!\!,\,\,\,
c^{(b c^l)^{\,2} b^{\!-1}} 
\big] ,\;\; k,l=1,2,\ldots
\big\rangle.
$$
\end{Example}

\subsection{The main construction of the Higman embedding}
\label{SU The main steps of the proof on Higman embedding theorem}

For free generators $a,b,c$ fix the free group $F_3=\langle a,b,c \rangle$ in addition to the above mentioned
$F_2=\langle b,c \rangle$. Denote by $b_i$ the conjugate 
$b^{c^i}$ for any $i\in \Z$.
Then for each function $f\!\in \E$ define the product 
$b_f=\cdots b_{-1}^{f(-1)}b_0^{f(0)}b_1^{f(1)}\cdots$
and the conjugate $a_f=a^{b_f}$\!.
Say, for $f=(5,2,-1)$ we have
$$a_f= 
a^{b_{0}^{5} b_{1}^{2} b_{2}^{-1}}
=
c^{-2} b c b^{-2} c
b^{-5}
\cdot a \cdot 
b^5 
c^{-1} b^{2} c^{-1} 
 b^{-1} c^2
.$$ 
For any subset $\B$ of $\E$
introduce the subgroup $A_\B = \langle a_f \mathrel{|} f \in \B\, \rangle$ in $F_3$.
In particular, for the zero set $\B=\Zz$ we get the subgroup 
$$
A_\Zz = \langle a_f \mathrel{|} f=(0)\,\rangle=
\langle a\rangle,
$$
and for the set $\B=\S$ we get the subgroup 
$$
A_\S = \langle a_f \mathrel{|} f\in \S\,\rangle = 
\langle 
c^{-1}\!b^{-(n+1)}c\,b^{-n}
\cdot a \cdot 
b^n c^{-1}\!b^{n+1}c \mathrel{|} n\in \Z\,\rangle.
$$
As is verified in Lemma 4.4 in \cite{Higman Subgroups in fP groups},
$A_\Zz$ and $A_\S$ are benign in 
$F_3$, and the respective $K$ and $L$ (check notation in \ref{SU Benign subgroups}) can easily be constructed for each of them.

The most part of \cite{Higman Subgroups in fP groups} is occupied by proofs for Theorem 3 and for Theorem 4 which set up the environment in which recursion is studied by group-theoretical means.
By Theorem 4 a set 
$\B$ is recursively enumerable in $\E$ 
if and only if 
$A_{\B}$ is benign in $F_3$,
and by Theorem 3  
$\B$ is recursively enumerable in $\E$ if and only if it belongs to $\BB$, i.e., it can be constructed from the basic sets $\Zz$ and $\S$ using the Higman operations \eqref{EQ Higman operations}.
This means we can start from benign subgroups $A_\Zz$ and $A_\S$, and as $\B$ is being built from $\Zz$ and $\S$ by some series of operations \eqref{EQ Higman operations}, the benign subgroup $A_\B$ is being constructed step-by-step.
Note that after Subsection~\ref{SU Extra auxiliary operations} we are free to also use the new auxiliary operations we suggested there.
    
\medskip
Each relation $r'_s$ we constructed in \ref{SU The universal generators} for our recursive
$2$-generator
group $T= \langle b,c \!\mathrel{|}\! R' \,\rangle$
can be written as $r'_s(b,c) = b^{n_0}c^{n_1} \cdots b^{n_{2m}}c^{n_{2m+1}}$ for some $m=m(s)$, and this presentation will be unique, if we also require $n_1,\ldots,n_{2m} \neq 0$. Thus, $r'_s$ can be ``coded'' by the sequence  of exponents $f_s=(n_0, n_1, \ldots ,n_{2m+1})$, and the elements $b_f=b_{f_s}$ and $a_f=a_{f_s}$ can be defined for these particular $f=f_s$. Say, for  $b^3 c b^{-1}c^2$ we have 
$f=(3,1,-1,2)$ and
$b_f=
b_0^{3} b_1 b_2^{-1} b_3^{2}$ \,with $a_f=a^{b_0^{3} b_1 b_2^{-1} b_3^{2}}$.

The set $\B=\{f_1, f_2,\ldots\}$ of all such sequences clearly is a subset of $\E$, and we can define the respective subgroup $A_\B
= \langle a_f \mathrel{|} f \in \B\, \rangle
=\langle  a_{f_s} \mathrel{|} s=1,2,\ldots\, \rangle$
in $F_3$. As we mentioned above, $A_\B$ is benign in $F_3$ if and only if $\B$
can be constructed from the basic sets $\Zz$ and $\S$ using the Higman operations.
This launches the following massive procedure in \cite{Higman Subgroups in fP groups}:
the set $\B$ is written as an output of a series of operations \eqref{EQ Higman operations} started from $\Zz$ and $\S$. For each step one of the following actions may be taken:

\begin{enumerate}
\item
$\B_1, \B_2$ are already given, 
and $\B_3$ is obtained from them by any of the \textit{binary} Higman operations $\iota,\; \upsilon$. Also given are the respective benign subgroups $A_{\B_1},\;A_{\B_2}$ in $F_3$, together with the respective groups $K_1, K_2$ and $L_1, L_2$  (see the remark about notation in \ref{SU Benign subgroups}).  
Then $A_{\B_3}$ also is benign, and can construct the respective $K_3$ and $L_3$.

\item
$\B_1$ is already given, and $\B_2$ is 
obtained from it by any of the \textit{unary} Higman operations
$\rho,\; \sigma,\; \tau,\; \theta,\; \zeta,\; \pi,\; \omega_m$ for $m=1,2,\ldots$
Also given are the respective benign subgroup $A_{\B_1}$ in $F_3$, together with the respective groups $K_1$ and $L_1$.  
Then $A_{\B_2}$ also is benign, and we have a mechanism  allowing us to construct the $K_2$ and $L_2$.
\end{enumerate}

This procedure eventually outputs our sequences set $\B$ together with $A_{\B}$, with the respective group $K$ and its subgroup $L$.

\smallskip
If for a group $G$ (or for groups of a given generic type) we are able to explicitly write the set $\B$, and are able to tell how $\B$ can be extracted from $\Zz$ and $\S$
by operations \eqref{EQ Higman operations}, then we have an embedding of $F_3$ into a finitely presented group $K$ with a finitely generated group $L$ such that $F_3 \cap L = A_\B$.

The final part of the Higman embedding is far shorter. 
By the proofs of Lemma 5.1 and Lemma 5.2 in~\cite{Higman Subgroups in fP groups} the normal closure 
$\bar R=\langle R' \rangle^{F_2}$ is benign in $F_2$ if and only if $A_{\B}$ is benign in $F_3$. 
The proofs of these lemmas also provide the finitely presented group $K$ with a finitely generated subgroup $L$ such that
$K$ embeds $F_2$, and also $F_2 \cap L = \bar R$.
Then ``the Higman Rope Trick'' (see the end of Section 5 in \cite{Higman Subgroups in fP groups}, 
p.~219 in \cite{Lyndon Schupp},
or \cite{Higman rope trick}) uses these $K$ and $L$ to embed $T= \langle b,c \!\mathrel{|}\! R' \rangle$, and thus also $G$, into a finitely presented group using a free product with amalgamation and a HNN-extension.

That is, \textit{if we are able to explicitly write $\B$ by the operations \eqref{EQ Higman operations}, then we can construct the explicit embedding of $G$ into a finitely presented group.}
This is what we are going to do in the rest of this note.

\subsection{Examples, the structure of sequences in $\B$}
\label{SU The structure of sequences}

Let us continue the earlier Example~\ref{EX embedding of free abellian into 2-generator group} by applying the constructions from previous subsection for the group $\Z^{\infty}$: 

\begin{Example}
\label{EX turn back}
The group $T$ in Example~\ref{EX embedding of free abellian into 2-generator group} has the relations:
$$
r''_{s}\,(b,c)=r''_{k,\,l}\,(b,c)=\big[
c^{(b c^k)^{\,2} b^{\!-1}} 
\!\!\!\!\!,\,\,\,
c^{(b c^l)^{\,2} b^{\!-1}} 
\big]
$$
$$
=
b\, c^{-k} b^{-1} c^{-k} b^{-1} c^{-1} b\, c^{k}  b\, c^{k-l} 
b^{-1} c^{-l} b^{-1} c^{-1} b\, c^{l}  b\, c^{l-k}
b^{-1}\!  c^{-k} b^{-1} \!  
c
b\, c^k   b  c^{k-l}
b^{-1}\!  c^{-l} b^{-1}   \!  
c
b\, c^l   b  c^l b^{-1}\!\!,
$$
$k,l=1,2,\ldots$\;
The respective sequence in $\E$ is: 
$$
f_s=f_{k,l}=
(
1,\, -k,\,  -1,\,  -k,\, -1,\,  -1,\,  1,\,  k,\,  1,\,  k\!-\!l,\,  
-1,\,  -l,\,  -1,\,  -1,\,  1,\, l ,\, 1,\,  l\!-\!k, 
\hskip15mm
$$
$$
\hskip50mm
-1,\,   -k,\,  -1,\,  1,\,  1,\,  k,\,    1 ,\,  k\!-\!l,\, 
-1,\,  -l,\,  -1,\,  1,\, 1, \, l,\,  1,\,  l,\, -1
).
$$ 
As we see, each $f_{k,l}$ is a sequence of length $35$ mostly filled by eleven entries $1$ and by eleven entries $-1$, with the following exceptions only: 
two entries are $k$;\;
three entries are $-k$;\;
three entries are $l$;\; 
two entries are $-l$;\;
two entries are $k\!-\!l$;\;
one entry is $l\!-\!k$ 
(the case with $k=l$ is not an exception, as $a_k$ commutes with itself, and we will just have some coordinates $k-k=0$ in the sequence above).
Denote the respective set of sequences as $\B=\{f_{k,l} ,\; k,l=1,2,\ldots \}$.

\smallskip
\textit{Can this $\B$ be constructed by a series of operations \eqref{EQ Higman operations}?} 
For now let us just simplify this question, postponing the full answer to Example~\ref{EX Back to the example for Z^infty}.
As we saw in Subsection~\ref{SU Extra auxiliary operations}, if $\alpha$ is any permutation of $\sup(\B)$, then $\B$
belongs to $\BB$ if and only if  $\alpha\B$ belongs to $\BB$.
In our case $\sup(\B)$ is in the set 
$\{0,1,\ldots,34\}$ of all $35$ indices.
It is trivial to find the permutation 
$$
\alpha =
(0)
\;
(1\,\,24\,\, 7\,\, 22\,\, 6)
\;
(2\,\, 11\,\, 30\,\, 9\,\, 32\,\, 10\,\, 14\,\, 3\,\, 25\,\, 33\,\, 29\,\, 8)\hskip50mm
$$
\vskip-8mm
$$\hskip45mm
(4\,\, 12\,\, 15 \,\,27\,\, 31\,\, 28\,\, 20\,\, 18\,\, 17 \,\,\,34\,\, 21 \,\,5 \,\, 13\,\; 16)
\;
(19  \,\;26) 
\;
(23)
$$ 
(we write the cycles of length $1$ also) that reorders the indices so that all the similar coordinates in $f_{k,l}$ are grouped, i.e., $\alpha  f_{k,l}$ starts by $1$ repeated eleven times, followed by $-1$ repeated eleven times, then followed by two times $k$, etc.:
\begin{equation}
\label{EQ grouped coordinates}
\alpha\, f_{k,l} =
\big(11\times 1,\;\; 
11\times -1,\;\; 
2\times k,\;\; 
3\times -k,\;\; 
3\times l,\;\; 
2\times -l,\;\; 
2\times(k\!-\!l),\;\; 
l\!-\!k
\big)
\end{equation}
where $11\times 1$ naturally means: $1$ repeated eleven times, etc...\;
\end{Example}

\begin{Remark}
\label{RE about permutations}
The above trick will be  used below repeatedly: applying a permutation $\alpha$ we may transform the set $\B$ to such an $\alpha\,\B$ in which coordinates are grouped in the manner of \eqref{EQ grouped coordinates}.
It is simpler to work with such an \textit{appropriately permuted} set $\alpha\,\B$ keeping in mind that $\B$
belongs to $\BB$ if and only if  $\alpha\B$ belongs to $\BB$, as we saw in Subsection~\ref{SU Extra auxiliary operations}. 
\end{Remark}

\begin{Example}
\label{EX embedding of free metabelian into 2-generator group}
Let $G = \big\langle a_1, a_2,\ldots \mathrel{|} \big[[a_k,a_l],[a_u,a_v] \big],\; k,l,u,v=1,2,\ldots \big\rangle$ be the free met\-abe\-l\-ian group $F_\infty(\mathfrak M)$ of countable rank in the variety of all metabelian groups $\mathfrak M$.
It is easy to deduce that this torsion-free group can be embedded into the $2$-generator group:
$$
T =\Big\langle b,c 
\;\mathrel{\Big|}\;
\Big[
\big[
c^{(b c^k)^{\,2} b^{\!-1}} 
\!\!\!\!\!,\,\,\,
c^{(b c^l)^{\,2} b^{\!-1}} 
\big] 
,\,
\big[
c^{(b c^u)^{\,2} b^{\!-1}} 
\!\!\!\!\!,\,\,\,
c^{(b c^v)^{\,2} b^{\!-1}} 
\big]
\Big] 
,\; k,l,u,v=1,2,\ldots
\Big\rangle.
$$
Then the respective appropriately permuted sequences (see the previous remark) will be: 
$$
\alpha\, f_{k,l,u,v} \!=\!
\big(40\times 1,\;\; 
49\times -1,\;
6\times k,\;
7\times -k,\;
6\times l,\;
6\times -l,\; 
7\times u,\; 
6\times -u,\;
6\times v,\; 
6\times -v,
\hskip10mm
$$
\vskip-8mm
$$
\hskip86mm
l-k,\;\; 
l-u,\;\;
v-u,\;\;
v-l,\;\;
k-l,\;\;
k-v,\;\;
u-v
\big)
$$
(we omit the routine calculations). 
\end{Example}

\begin{Example}
\label{EX embedding of rational group}
The additive group of rational numbers $\Q$ has a presentation $\big\langle a_1, a_2,\ldots \mathrel{|} a_s^s=a_{s-1},\; s=2,3,\ldots \big\rangle$
where a generator $a_i$ corresponds to the fraction ${1 \over i!}$ with $i=2,3,\ldots$\,
\cite{Johnson}.
In Example~3.5 in \cite{Embeddings with universal words} we gave the embedding of $\Q$ into the $2$-generator group:
$$
T =\big\langle b,c 
\;\mathrel{|}\;
(c^s)^{(b c^s)^{\,2} b^{\!-1}} 
(c^{-1})^{(b c^{s-1})^{\,2} b^{\!-1}}
\!\!,\;\; s=2,3,\ldots
\big\rangle
=\big\langle b,c 
\;\mathrel{|}\;
(c^s)^{b c^s b c} 
(c^{-1})^{b c^{s-1}b}
\!\!,\;\; s=2,3,\ldots
\big\rangle.
$$
The respective appropriately permuted sequences (see Remark~\ref{RE about permutations}) then are: 
$$
\alpha\, f_{s} \!=\!
\big(6\times 1,\;\; 
6\times -1,\;
2\times s,\;
2\times -s,\;
1\!-s,\;
2\times (s\!-\!1)
\big),
$$
$s=2,3,\ldots$ \;
(the omitted calculations are easy to verify).
\end{Example}

\begin{Remark}
	\label{RE embedding of rational group by Belk Hyde Matucci}
In 1999 Bridson and de la Harpe posed
in the Kourovka notebook \cite{kourovka}
Problem 14.10 in which they grouped a few questions as a ``well-known problem''. 
The questions mainly concern explicit embeddings of some countable groups into finitely generated or finitely presented groups. In particular, one of the points of Problem 14.10 (a) asks to find an explicit embedding of $\Q$ into  a ``natural'' finitely presented group.

As the main steps outlined in Section~\ref{SE The main steps of the embedding} show, we are able to explicitly embed a  recursive group $G$ into a finitely presented group, as soon as we have the explicit embedding of $G$ into the respective $2$-generator group $T=T_G$, have the set $\B$ of integer sequences corresponding to defining relations of $T_G$, and also are able to construct $\B$ from the sets $\Zz$ and from $\S$ using the Higman operations \eqref{EQ Higman operations}.
Example~\ref{EX embedding of rational group} directly provides $T$ for $\Q$, and it gives $\B$  by means of $\alpha\, f_{s}$. 

The $H$-machine of Section~\ref{SE H-machine} shows how to easily write down the operations \eqref{EQ Higman operations}, if we know $\B$.
That is, a group answering Problem 14.10 of Bridson and de la Harpe \cite{kourovka} can be constructed by a series of free constructions matching to the series of Higman operations. Of course, the question is if that group can be called a ``natural'' finitely presented group...

Recently a direct solution to the problem of Bridson and de la Harpe was found by Belk, Hyde and Matucci in \cite{Belk Hyde Matucci}. Moreover, one of the remarkable finitely presented groups constructed by them is the group $T\!\mathcal{A}$ which is $2$-generator and also simple \cite{Belk Hyde Matucci}.
\end{Remark}

\begin{Example}
\label{EX embedding of Pruefer group}
The quasicyclic Pr\"ufer $p$-group $G=\Co_{p^\infty}$ can be presented as:
$$
G = \big\langle a_1, a_2,\ldots \mathrel{|}
a_1^p,\;\;\, a_{s+1}^p\!=a_s
,\;\; s=1,2,\ldots \big\rangle
$$ 
where each $a_i$ corresponds to the 
primitive $(p^i)$'th root $\varepsilon_i$ of unity \cite{Kargapolov Merzljakov}.
As we found in Example~3.6 in \cite{Embeddings with universal words}, this group can be embedded into the $2$-generator group:
$$
T =\big\langle b,c 
\;\mathrel{|}\;\;
\big(c^{(b c)^{\,2}\, b^{\!-1}} 
 (c^{-1})^{b}\big)^p\!\!,\;\;\;\;
\big(c^{(b c^{s+1})^{\,2}\, b^{\!-1}} 
 (c^{-1})^{b}\big)^p 
c^{b} (c^{-1})^{\,(b c^s)^{\,2}\, b^{\!-1}}
\!\!\!,\;\;\; s=1,2,\ldots
\big\rangle.
$$
From the first single relation we get the appropriately permuted sequence (see Remark~\ref{RE about permutations}):
$$
\alpha\, f_{0} \!=\!
\big(
(5p\!+\!2)\times 1,\;\; 
5p\times -1,\;\;\;
(p-1)\times 2,\;\;
p\times -2
\big).
$$
And from the remaining relations we get the respective appropriately permuted sequences: 
$$
\alpha'\, f_{s} \!=\!
\big(
(3p\!+\!4)\times 1,\;\; 
(3p\!+\!3)\times -1,\;\;\;
s,\;\;
-s,\;\;
p\times (s+1),\;\;
p\times (-s-1)
\big),
$$
$s=2,3,\ldots$ \;
(the calculations are omitted).
Clearly, $\alpha'\neq \alpha$.
\end{Example}

Examples similar to 
Example~\ref{EX turn back} and 
Example~\ref{EX embedding of free metabelian into 2-generator group} are easy to construct 
for free \textit{soluble} groups, 
for free \textit{nilpotent} groups and, more generally, for other types of groups defined by commutator-based identities.

Further, since any \textit{divisible} abelian group is a direct product of copies of $\Q$ and of some $\Co_{p^\infty}$,
it is not hard to use
Example~\ref{EX embedding of rational group} and 
Example~\ref{EX embedding of Pruefer group} 
to get sequences of similar formats for them also. Moreover, every abelian group is a subgroup in an abelian divisible group, so we get similar sequences for embeddings of any countable abelian group (provided that its embedding into a countable divisible abelian group is constructively, effectively given).

\begin{Remark}
\label{RE numerous groups}
We could continue collection of examples with the same features, but it already seems to be clear that there are numerous groups for which the respective sequence sets have certain similar ``format''. Namely: 
\begin{enumerate}
\item \vskip-0.7mm Some coordinates in them have a \textit{fixed} value (or one of pre-given fixed values). Say, the initial $0$'th coordinate is equal to $1$ in each sequence $\alpha f_{k,l}$ in Example~\ref{EX turn back}.

\item Some coordinates can accept \textit{any} integer values $k$, like the $22$'nd coordinate $k$ in the sequence $\alpha f_{k,l}$.

\item Some coordinates are \textit{duplicates} of certain other coordinates. Say, the 
$1$'st,
$2$'nd, ...,
$10$'th coordinates in $\alpha f_{k,l}$ all are the duplicates of the $0$'th coordinate $1$. And also 
the $23$'rd coordinate $k$ is the duplicate of the $22$'nd coordinate $k$.

\item Some coordinates are the \textit{opposites} of certain other coordinates. Say, the 
$11$th coordinate $-1$ in $\alpha f_{k,l}$ is the opposite of the 
$10$'th coordinate $1$. Also, the $27$'th coordinate $-k$
is the opposite of the 
$26$'th coordinate $k$.

\item And some coordinates are obtained from other coordinates by \textit{arithmetical operations}.
Say, the 
$33$'th coordinate $k\!-\!l$ in $\alpha f_{k,l}$ is the difference of the 
$22$'nd coordinate $k$ and of the  the $28$'th coordinate $l$.
\end{enumerate}
\end{Remark}

As we see now, construction of a set $\B\in \BB$ by operations \eqref{EQ Higman operations} in many cases can be reduced to the question: 
can we build a ``machine'' which  \textit{constructs $\B$ 
by   performing the five operations listed above}, i.e., by assigning fixed pre-given values to some coordinates, then copying those values to other coordinates, then assigning the opposites, the sums or differences of those values to some other coordinates?
\textit{If yes, then constructive Higman embeddings are available for the considered types of groups. }

In the next section we will step by step collect a positive answer to this question.
The reader may skip to Example~\ref{EX Back to the example for Z^infty} to see an application of the method.

\section{The $H$-machine}
\label{SE H-machine}

\noindent
This is the main section of this note, and its  objective is to show that $\BB$ contains some general kinds of subsets of $\E$
which can be constructed by generic operations outlined in Remark~\ref{RE numerous groups}.
The reader not interested in the routine of proofs may skip the details below.

\subsection{Construction of sum of subsets with disjoint supports}
\label{SU Construction of sum of subsets with disjoint supports}

For definition of the sum of subsets from $\BB$ and of other auxiliary operations we refer to \ref{SU Extra auxiliary operations}. 

\begin{Lemma} 
\label{LE closed under sum}
If the sets
$\mathcal B_k$, $k=1,\ldots, m$, all belong to $\BB$, and their supports 
$S_k=\sup(\mathcal B_k)$ are finite pairwise disjoint sets, then the sum 
$\mathcal B_1+ \cdots + \mathcal B_m$
also belongs to $\BB$.
\end{Lemma}

\begin{proof} The set
$\mathcal B^*_1 = (\zeta_{\!S_2} \!\cdots \zeta_{\!S_m}) \mathcal B_1$
clearly consists of all functions from 
$\mathcal B_1$ with all coordinates from
$S_2,\ldots,S_m$ liberated. This can be achieved by applying some operations $\zeta_i$ for finitely many times, so $\mathcal B^*_1$ belongs to $\BB$.
In a similar way we define the sets
$\mathcal B^*_2,\ldots, \mathcal B^*_m$ in $\BB$. It is easy to see that 
$\mathcal B_1+ \cdots + \mathcal B_m=\iota_m(\mathcal B^*_1,\ldots,\mathcal B^*_m)$.
\end{proof}

The analog of this lemma could be proved for the case of infinite supports, but we restrict to this case for simplicity.

\subsection{Construction of $(n)$ with restrictions on $n$}
\label{SU Construction of with restrictions on n}

Denote by $\mathcal B_{+}=\big\{(n)  \mathrel{|} n=1,2,\ldots \big\}$
the set of all functions with a single positive coordinate, and by 
$\mathcal B_{-} = \big\{(n)  \mathrel{|} n=-1,-2,\ldots \big\}$
the set of all functions with a single negative coordinate.
Their union $\mathcal B_{\pm}=\big\{(n)  \mathrel{|} n\in \Z \backslash \{0\}\big\}$  is
the set of all functions with a single non-zero coordinate.

\begin{Lemma} 
\label{LE three types}
The sets
$\mathcal B_{+}$, 
$\mathcal B_{-}$
and\,
$\mathcal B_{\pm}$ 
belong to $\BB$. 
\end{Lemma}

\begin{proof} 
$\mathcal A_1=\omega_2 \upsilon(\zeta_{\!1} \Zz , \tau\S)$
clearly consists of functions $g$ in which for every $i\in \Z$ the subsequence 
$t_i=\big(g(2i),\,g(2i+1)\big)$ is either of type $(0,n)$ or of type $(n,n\!-\!1)$, with $n\in \Z$. 
For any \textit{even} $n=2,4,\ldots$ we can apply $\omega_2$ to the pairs $(n,n-\!1), (n\!-\!2, n\!-\!3),\ldots, (2,1) \in \tau\S$  
to construct in $\mathcal A_1$ the sequence $g=(n,n\!-\!1,\ldots,1)$.
The sequence $g'=(0,n,n-1,\ldots,1,0)$ can be built by the pair $(0,n)\in \zeta_{\!1} \Zz$ and the pairs $(n-1,n-2),\ldots, (3,2), (1,0) \in \tau\S$.
Clearly, $\sigma^{-1} g' = g$, and so $g\in \mathcal A_2 = \iota(\mathcal A_1,\sigma^{-1} \mathcal A_1)$.
In a similar manner we discover in $\mathcal A_2$ all the functions $g=(n,n\!-\!1,\ldots,1)$ for \textit{odd} $n=1,3,\ldots$
This time $g$ is constructed by the pairs 
$(n,n-\!1), (n\!-\!2, n\!-\!3),\ldots, (1,0) \in \tau\S$, and 
$g'$ can be built by $(0,n)\in \zeta_{\!1} \Zz$ with $(n-1,n-2),\ldots, (2,1) \in \tau\S$.
Thus, for any $n=1,2,\ldots$ the set $\mathcal A_2$ contains a function $g$ with the property $g(0)=n$.

Let us show that  $g(i)<0$ is impossible for any $g\in \mathcal A_2$.
Assuming the contrary, suppose the \textit{least} coordinate $g(k)<0$ of $g$ is achieved at some index $k$. 
If $k$ is \textit{even}, then the pair
$\big(g(k), g(k+1)\big)$ in $g\in \mathcal A_1$ has to be either of type $(n, n\!-\!1)\in \tau \S$ (which is impossible as $g(k-1)\nless g(k)$) or of type $(0,n)\in \zeta_1 \Zz$ (which is impossible as  $g(k)=0 \nless  0$). 
And if $k$ is \textit{odd}, then the pair
$\big(g(k), g(k+1)\big)$ in $g\in \sigma^{-1} \mathcal A_1$ has to be either of type $\sigma^{-1}(n, n\!-\!1)$ or of type $\sigma^{-1}(0,n)$ (which both again are impossible). 
%

Next let us exclude those functions $g\in \mathcal A_2$ for which $g(0)=0$. 
Clearly, 
$\mathcal A_3 = \pi_1 \pi' \tau \S$ is the set of all those functions from $\E$ which coincide with $(n,n-\!1)$ on indices $0,1$, and which may have any coordinates elsewhere.
Then $g(0)>0$ for each $g\in  \iota(\mathcal A_2, \mathcal A_3)$, and the extract $\epsilon_0\, \iota(\mathcal A_2, \mathcal A_3)$ is the set 
$\mathcal B_{+}$.

In a similar way can construct 
$\mathcal B_{-}$.\;
And the union of the above is $
\mathcal B_{\rm \pm}
=\upsilon(\mathcal B_{\rm +}, \mathcal B_{\rm -})$.
\end{proof}

The reader may compare the above proof with the argument of Lemma 2.1 in~\cite{Higman Subgroups in fP groups}.

\medskip

Let $\mathcal B_{k+}$ or $\mathcal B_{k-}$ denote the set of all $(n)$ for which $n>k$ or $n<k$, respectively. 

\begin{Lemma} 
\label{LE (n_1), ..., (n_k)}
For any integers $k \in \Z$ the sets
$\mathcal B_{k+}$ and $\mathcal B_{k-}$ 
belong to $\BB$.
\end{Lemma}

\begin{proof} The set
$\mathcal D_1=\iota (\tau \S, \zeta \sigma \mathcal B_{+})$ consists of  all pairs $(n+1,n)$ for $n=1,2,\ldots$
Then 
$\mathcal B_{1+}$ is the extract 
$\epsilon_0\,\mathcal D_1$.
By induction we construct in $\BB$ the set $\mathcal D_k=
\iota (\tau \S, \zeta \sigma \mathcal B_{(k-1)+})
$, and the extract
$\mathcal B_{k+}=\epsilon_0\,\mathcal D_k$.
The case of $\mathcal B_{k-}$ is discussed analogously.
\end{proof}

For an integer $n\in \Z$ denote by 
$\mathcal N_n$ the set $\big\{(n)\big\}$ consisting of a single sequence $(n)$ of length $1$.
More generally, denote $\mathcal N_{n_1,\ldots,n_k}=\big\{(n_1), \ldots, (n_k)\big\}$ the set consisting of $k$ functions of the above type.

\begin{Lemma} 
\label{LE (n_1), ..., (n_k)}
For any fixed integers $n_1,n_2,\ldots,n_k \in \Z$ the set
$
\mathcal N_{n_1,\ldots,n_k}
$ 
belongs to $\BB$.
\end{Lemma}

\begin{proof}
It is clear that 
$\mathcal N_1=\big\{(1)\big\}$ is in $\BB$, for, 
$\tau \S = 
\big\{(n\!+\!1,n) \mathrel{|} n\in \Z\big\}
$, and so $\mathcal N_1=\iota (\tau \S, \zeta \Zz)$ consists of  $(0+1,0)=(1,0)=(1)$ only. 
Similarly 
$\mathcal N_2=\big\{(2)\big\}$ is in $\BB$ because 
$\mathcal N_2=\epsilon_0\, \iota (\tau \S, \zeta \sigma \mathcal N_1)$. By induction we construct all the $\mathcal N_3,\; \mathcal N_4, \ldots$
The sets
$\mathcal N_{-1},\; \mathcal N_{-2}, \ldots$ can be obtained in a similar way. 
Finally, $\mathcal N_0=\big\{(0)\big\} = \Zz$.
Taking the union of the required one-element sets 
$\big\{(n_1)\big\},\ldots,\big\{(n_k)\big\}$ we finish the proof.
\end{proof}

\subsection{Duplication of the last term}
\label{SU Duplicating n}

Let $\mathcal B \in \BB$ be any set of functions $g$ 
which are zero after the $k$'th coordinate, i.e.,
$g(i)=0$ for each $i>k$. Then by Higman operations we can ``duplicate'' the $k$'th coordinate in all $g\in \mathcal B$.
More precisely, for each $g\in \mathcal B$ let 
$g'$ be defined as:
$g'(i)=g(i)$ for all $i\neq k\!+\!1$,\;
$g'(k\!+\!1)=g(k)$.
In this notation define $\mathcal B_{k,k} =\{g' \mathrel{|}g\in \mathcal B\}$.

\begin{Lemma} 
\label{LE duplicate k}
Let $\mathcal B \in \BB$ be a set of functions $g$ which are zero after the $k$'th coordinate. 
Then the set $\mathcal B_{k,k}$ also belongs to $\BB$. 
\end{Lemma}

\begin{proof}
For simpler notation assume $k=0$ as the general case can be deduced to this by shifting $\mathcal B$ by $\sigma^{-k}$\!,\, and then shifting back by $\sigma^{k}$ after duplication of the $0$'th coordinate.

Denote $\mathcal F_1$ to be the set of all functions from $\mathcal B$ with the $1$'st and $2$'nd coordinates liberated, i.e., $\mathcal F_1 =
\zeta_{\!1,2}\,\mathcal B$.
Let $\mathcal F_2$  be the set of all functions $g\in \E$ in which $g(1)=g(0)+1$, and the $2$'nd coordinate together with \textit{all} the negative coordinates are liberated, i.e., 
$\mathcal F_2
=\pi'\zeta_{\!2}\,\S$. 
Let $\mathcal F_3$  be the set of all functions $g$ in which $g(2)=g(1)-1$, and the $0$'th coordinate together with \textit{all} the negative coordinates are liberated, i.e., 
$\mathcal F_3
=\pi'\zeta \sigma \tau\,\S
=\pi'_{1} \sigma \tau\,\S
$.  
Then the $2$'nd and $0$'th coordinates of each function from $\mathcal F_4 = \iota_3 
(\mathcal F_1 ,\, \mathcal F_2 ,\, \mathcal F_3)
$ are equal. 

To get the duplicated set $\mathcal B_{0,0}$ it remains to swap the $2$'nd and $1$'st coordinates in $\mathcal F_4$, and then to erase the new $2$'nd coordinates.
Namely, 
set 
$\mathcal F_5 
=\zeta_2 \tau_{1,2}\mathcal F_4$ and 
$\mathcal F_6 
=\pi'_2\Zz
$, and take the intersection
$\mathcal B_{0,0} =\iota (\mathcal F_5 ,\, \mathcal F_6)$.
\end{proof}

\subsection{Construction of the pairs $(n,-n)$}
\label{SU Construction of (n,-n)}

Denote by $\mathcal B_{+,-} =
\big\{(n,-n)  \mathrel{|} n=1,2,\ldots \big\}
$  
the set of all couples $(n,-n)$ with $n=1,2,\ldots$
The objective of this subsection is to prove:

\begin{Lemma} 
\label{LE (n,-n)}
	The set
	$\mathcal B_{+\,-}$ belongs to $\BB$.  
\end{Lemma}

Our proof will follow from a series of steps, cases, examples below.

\medskip

The set 
$\mathcal L_1 
=\zeta_{2,3}\Zz
$
can be interpreted as the set of all $4$-tuples
\begin{equation}
\label{EQ first tuple}
(0,0,\;m,n)
\end{equation}
with $m,n\in \Z$. 
Next, 
$\mathcal C_1 = \tau_{1,2}\tau\S$
can be interpreted as the set of all $4$-tuples
$$(m,0,\;\;m\!-\!1,0),$$
while 
$\mathcal C_2 = \tau_{1,2} \sigma^2 \S$
can be interpreted as the set of all $4$-tuples
$$(0,n,\;\;0,n\!+\!1).$$
Then the sum 
$\mathcal L_2 = \mathcal C_1+\mathcal C_2
$
is the set of all $4$-tuples
\begin{equation}
\label{EQ second tuple}
(m,n,\;\; m\!-\!1,n\!+\!1).
\end{equation}
The set  $\mathcal L_3=\omega_4 \upsilon(\mathcal L_1 , \mathcal L_2 )$ consists of $g\in \E$ in which for every $i\in \Z$ the subsequence 
\begin{equation}
\label{EQ 4-tuple}
t_i\!=\!\big(g(4i),\,g(4i\!+\!1),\,g(4i\!+\!2),\,g(4i\!+\!3)\big)
\end{equation}
is of type \eqref{EQ first tuple} or of type \eqref{EQ second tuple} 
(not ruling out the zero $4$-tuple which  is of type \eqref{EQ first tuple} for $m=n=0$).
Define a set $\mathcal M = \iota(\mathcal L_3,\sigma^{-2} \mathcal L_3)$.

\medskip

\textbf{Step 1.} 
Start by showing that if $g\in \mathcal M$, then $g(k)\ge 0$ for any $k=4i,\; 4i+2$ with $i\in \Z$.
Assume the contrary: $m=g(k)< 0$ for some $k$ of one of the above types.

If $k=4i$, i.e, $m$ is the initial term of the $4$-tuple $t_i$ in \eqref{EQ 4-tuple}, then $t_i$ is of type \eqref{EQ second tuple} because the tuples of type \eqref{EQ first tuple} have to start by a zero. Thus we have $g(4i+2)=m-1<0$.

Next assume $k=4i+2$.
As $g\in \sigma^{-2} \mathcal L_3$,
there exists a $g' \in \mathcal L_3$ such that  $\sigma^{-2} g' = g$.
Then $g'(4(i+1)) = g'(4i+2+2) = g(4i+2)=m$, i.e., the next 
$4$-tuple of $g'$ also starts by  negative number $m$, and has to be of type \eqref{EQ second tuple}. 
But then $g'(4(i+1)+2) = m-1$, and so
$g(4i+2+2) = g(4(i+1)) = m-1<0$, i.e., the $(i+1)$'st sequence $t_{i+1}$ in $g$ starts by $m-1$. 

We got that for any $k=4i$ and $k=4i+2$ from $g(k)<0$ it follows $g(k+2)<0$, $g(k+4)<0$, etc., because 
$g(k+2)= m-1$, $g(k+4)= m-2$, etc.
This leads to a contradiction as $g\in \E$ cannot have infinitely many non-zero coordinates.

In a similar way we show that $g(k)\le 0$ for any $k=4i+1, 4i+3$ for $i\in \Z$.

\begin{Example} 
\label{EX two examples with 8-tuples}
Consider two functions $g\in \mathcal M$ of above types.
First, the function 
\begin{equation}
\label{EQ first example with 4}
g=(
4, -4,
3, -3;\;\;\;\;
2, -2,
1, -1
)
\end{equation}
is constructed by two $4$-tuples of type \eqref{EQ second tuple}, and it can be presented as $g=\sigma^{-2} g'$ for
$$
g' =(
0, 0, 
4, -4;\;\;\;\;
3, -3, 
2, -2;\;\;\;\;
1, -1,
0,  0
)
$$ 
which is constructed by one $4$-tuple of type \eqref{EQ first tuple} and 
two $4$-tuples of type \eqref{EQ second tuple}.

Next consider another function
\begin{equation}
\label{EQ second example}
g=(
0, 0, 
4, -4;\;\;\;\;\;
3, -3,
2, -2;\;\;\;\;\;
1, -1,
0, 0
)
\end{equation}
constructed by one $4$-tuple of type \eqref{EQ first tuple}, and two 
$4$-tuples of type \eqref{EQ second tuple}, and this $g$ can be presented as $g=\sigma^{-2} g'$ for the function
$$
g' =(
0,0,0,0;\;\;\;
4, -4,
3, -3;\;\;\;
2, -2,
1, -1
)
$$ 
which is constructed by 
two $4$-tuples of type \eqref{EQ second tuple} (and the zero $4$-tuples, of course).

Observe that in these two functions we took  $4$ and $-4$ to be the \textit{opposites} of each other, ignoring the case of a tuple, say, $(4,0,-5,0)$.
We will cover that issue later.
\end{Example}

\medskip
\textbf{Step 2.} 
We see that from any positive $g(k)$ \textit{a chain} of positive, descending coordinates 
$g(k), g(k+2), g(k+4), \ldots$ starts
for a $k=4i$ or $k=4i+2$.
%
How may this chain end?

\textit{Case 2.1.}
The chain achieves $1$ (its \textit{last} positive coordinate) at some index $4j$, i.e., in the first half of some $4$-tuple $t_j$, like in \eqref{EQ second example}, then the term $g(4j+2)$ automatically is $1-1=0$. Starting from the term $g(4j+4)$ in $t_{j+1}$ we may have either zeros, or a new chain may begin from there.

\textit{Case 2.2.}
The chain achieves $1$ at some index $4j+2$, i.e., in the second half of some $4$-tuple $t_j$, like in \eqref{EQ first example with 4}. Then the next term $g(4j+2+2)=g(4(j+1))$ (which is $0$ and which lies in the next tuple $t_{j+1}$) may have two potential ways to occur:
either the next tuple is of type \eqref{EQ first tuple}, i.e., it starts by two zeros, and after them  we may have either zeros, or a new chain may begin there;\;
or the next tuple is of type \eqref{EQ second tuple} with an initial term $g(4(j+1))=0$. But then the $(4(j+1)+2)$'nd term of that tuple has to be $0-1=-1$. Since negative values are ruled out for such coordinates, that is impossible.

\smallskip 
The analogs of these arguments hold for negative, ascending chains $g(k), g(k+2), g(k+4), \ldots$ starting at some 
$g(k)$ for a $k=4i+1$ or $k=4i+3$.
Namely the \textit{last} negative term $-1$ of such a term is achieved:

\textit{Case 2.3.}
either at some $4j+1$, i.e., in the first half of some $4$-tuple, like in \eqref{EQ first example with 4}, then the next term $g(4j+3)$ automatically is $1-1=0$,

\textit{Case 2.4.}
or $-1$ is achieved at some index $4j+3$, i.e., in the second half of some $4$-tuple, like in \eqref{EQ second example}. Then the next next tuple may be of type \eqref{EQ first tuple} only, i.e., it starts by two zeros.

\medskip
\textbf{Step 3.} 
The key feature of this construction is that the two chains we discuss (the ascending and the descending chains residing inside some consecutive $4$-tuples)  \textit{terminate simultaneously}, i.e.,
the last $4$-tuple $t_j$ 
either ends by $(1,-1)$\;
(i.e., $t_j=(2,-2, 1,-1)$  
or  $t_j=(0,0, 1,-1)$), 
or $t_j=(1,-1,\; 0,0)$.
Assume the contrary, and arrive to contradiction in all cases occurring.

\textit{Case 3.1.}
Assume $t_j$ ends by $(m,-1)$ for an $m\ge 2$. Then by Case 2.2 above the next
$4$-tuple $t_{j+1}$ need start with two zeros. We get a contradiction because $m-1\neq 0$.

\textit{Case 3.2.}
Assume $t_j$ ends by $(0,-1)$\; (that is, $m=0$ in terms of the previous case).
Since $g=\sigma^{-2}g'\in \mathcal L_3$,
the $(j+1)$'st $4$-tuple in $g'$ starts by $(0,-1)$. Then that $4$-tuple in $g'$ is of type \eqref{EQ second tuple}, i.e., it ends by $(0-1,-1+1)=(-1,0)$. So $t_{j+1}$ in $g$ starts by $(-1,0)$, which is a contradiction as $g(4(j+1))=-1$ cannot be negative.

\textit{Case 3.3.}
Assume the last $4$-tuple $t_j$ is
$(m,-1,\;0,0)$ with $m\ge 2$ or $m=0$.
Since $-1\neq 0$, then $t_j$ is of type \eqref{EQ second tuple}. Then its $2$'nd term is $m-1$ which is impossible as $m-1\neq 0$.

\medskip
We get that whenever a $g\in \mathcal M$ contains a couple $(m,n)$ with a positive $m$ and a negative $n$, we have $n=-m$ (see the remark at the end of Example~\ref{EX two examples with 8-tuples}).
In particular, if for some $g\in \mathcal M$ we have $g(0)>0$ and $g(2)<0$, then 
$g(2)=-g(2)$. Clearly, for any positive $n$ we can build an $g\in \mathcal M$ with $g(0)=n$ and $g(2)=-n$.

\medskip
\textbf{Step 4.} 
Denote by $\mathcal C_3 = \pi_2 \pi'\mathcal M$ the set of all functions $g\in \E$ which coincide with some $(n,-n)$ with $n=0,1,\ldots$, and which have any coordinates elsewhere.
$\mathcal C_4 =
\mathcal B_{+} + \sigma \mathcal B_{-}
$ can be interpreted as the set of all couples $(m,n)$ with positive $m$ and negative $n$.
Then $\mathcal B_{+,-} = \iota(\mathcal C_3, \mathcal C_4)$ is the set of all couples $(n,-n)$ with $n=1,2,\ldots$

\medskip
Thus, Lemma~\ref{LE (n,-n)} is fully argued.
\medskip

If needed, we can easily get the analogs of Lemma~\ref{LE (n,-n)} not only for the couples  $(n,-n)$ for \textit{all} $n=1,2,\ldots$ but for, say, $n=k,k+1,\ldots$, or for $n$ from a given finite set only.

\subsection{Construction of the triples $(p,\;q,\;p-q)$ and $(p,\;q,\;p+q)$}
\label{SU Construction of the triples (...)}
Assume the set $\mathcal P$ consists of some pairs $(p,q)$. In this subsection we show that if $\mathcal P$ belongs to $\BB$, then the set consisting of all triples 
$(p,\;q,\;p-q)$ and the set consisting of al triples $(p,\;q,\;p+q)$ also belong to $\BB$.

For a fixed pair $(p,q)$ denote by $\mathcal P_1$ the set of $8$-tuples of the following types:
\begin{equation}
\label{EQ type 1}
(p,q,q,0,\;\;0,0,0,0),
\end{equation}
\begin{equation}
\label{EQ type 2}
(0,0,0,0,\;\;p,q,p,n),
\end{equation}
\begin{equation}
\label{EQ type 3}
(p,q,m,n,\;\;p,q,m\!-\!1,n\!-\!1),
\end{equation}
\begin{equation}
\label{EQ type 4}
(p,q,m,n,\;\;p,q,m\!+\!1,n\!+\!1),
\end{equation}
for any $m,n\in \Z$, 
together with the zero function which we can interpret as the $8$-tuple 
$(0,0,0,0,\;\, 0,0,0,0)$.

The set of $8$-tuples of type \eqref{EQ type 1} is in $\BB$, since
we can apply Lemma~\ref{LE duplicate k} to couples $(p,q)$ to duplicate the coordinate $q$.
Similarly, the set of $8$-tuples of type \eqref{EQ type 2} also are in $\BB$.
The $8$-tuples of type \eqref{EQ type 3} can be obtained as follows: the set of all $8$-tuples of type $(0,0,m,n,\;\;0,0,m\!-\!1,n\!-\!1)$
can be obtained using a permutation $\alpha$ of the set 
$\sigma^2 \tau \S + \sigma^3  \tau \S$.
Then we take the sum of that set and the set of all $8$-tuples
$(p,q,0,0,\;\;p,q,0,0)$.
The case of $8$-tuples of types  \eqref{EQ type 4} is covered in a similar way.
This means the combined set $\mathcal P_1$ of \textit{all} $8$-tuples of types  \eqref{EQ type 1}--\eqref{EQ type 4} is in $\BB$.

Thus, the intersection $\mathcal P_2 = \iota(\mathcal P_1,\; \sigma^{-4}\mathcal P_1)$ also is in $\BB$.
Using Higman operations on $\mathcal P_2$ we can construct $(p,\;q,\; p-q)$. Let us first explain the idea by simple examples:

\begin{Example}
\label{EX construction of (p, q, p-q)}
Let $p=6$ and $q=2$. The sequence
$$
g=(
6, 2, 6, 4, \;
6, 2, 5, 3;\;\;\;\;
6, 2, 4, 2, \;
6, 2, 3, 1;\;\;\;\;
6, 2, 2, 0, \;
0, 0, 0, 0
)
$$
is constructed by two $8$-tuples of type \eqref{EQ type 3} and by one $8$-tuple of type \eqref{EQ type 1}. And $g$ can be presented as $g=\sigma^{-4} g'$ for
$$
g' =(
0, 0, 0, 0,
6, 2, 6, 4;\;\;\;\;
6, 2, 5, 3, 
6, 2, 4, 2;\;\;\;\;
6, 2, 3, 1,
6, 2, 2, 0
)
$$ 
which is constructed by one $8$-tuple type \eqref{EQ type 2} and 
two $8$-tuples of type \eqref{EQ type 3}.
Note that the $3$'rd coordinate in $g$ is $p-q=6-2=4$.

Yet another function
$$
g=(
0, 0, 0, 0,\;
6, 2, 6, 4;\;\;\;\;
6, 2, 5, 3, \;
6, 2, 4, 2;\;\;\;\;
6, 2, 3, 1, \
6, 2, 2, 0
)
$$
is constructed by one $8$-tuple of type \eqref{EQ type 2} and by two $8$-tuples of type \eqref{EQ type 3}. And $g$ can be presented as $g=\sigma^{-4} g'$ for
$$
g' =(
0, 0, 0, 0,\;
0, 0, 0, 0;\;\;\;\;
6, 2, 6, 4,\;
6, 2, 5, 3;\;\;\;\;
6, 2, 4, 2,\;
6, 2, 3, 1;\;\;\;\;
6, 2, 2, 0,\;
0, 0, 0, 0
)
$$ 
which is constructed by two $8$-tuples of type \eqref{EQ type 3} and 
one $8$-tuple of type \eqref{EQ type 1} (and the zero sequences, of course).
Note that the $11$'th coordinate in $g$ is $p-q=6-2=4$.

As we see, using Higman operations it is easy to obtain the triple $(6,\, 2,\, 4)$ from any of the functions $g$ constructed above. Using some loose wording we could say that we ``mimic'' the arithmetical operation $6-2=4$ by means of Higman operations (in the sense that we were able to build the triple $(6,\; 2,\; 6\!-\!2)$).
We did this using some descending chains of coordinates (at indices $4k+2, 4k+6, \ldots$) starting by $6$ and ending by $2$. 

The purpose of the numbers $6,2$ standing at some indices $8k, 8k\!+\!1$ or $8k\!+\!4, 8k\!+\!5$ is the following. Besides the pair $(6,2)$ our set $\mathcal P$ may also contain another pair, say, $(9,1)$. We want to construct in $\mathcal P$ the triple $(6,\, 2,\, 4)$ \textit{without} 
adding the unnecessary triple $(6,\, 1,\, 5)$ into $\mathcal P$. That is, we need a descending chain starting by $6$ and ending by $2$ (but \textit{not} by $1$). So those numbers $6,2$ guarantee that we concatenate $8$-tuple corresponding to the \textit{same} pair $(6,2)$ only.

Observe that $p\ge q$ in each of above examples. When $p< q$, then we could build \textit{ascending} chains using $8$-tuples of type \ref{EQ type 4}. Say, if $p=3$ and $q=9$, the sequence
\begin{equation*}
g=(
3, 9, 3, -6, 
3, 9, 4, -5;\;\;
3, 9, 5, -4, 
3, 9, 6, -3;\;\;
3, 9, 7, -2, 
3, 9, 8, -1;\;\;
3, 9, 9,  0, 
0, 0, 0,  0
)
\end{equation*}
is constructed by three $8$-tuples of type \eqref{EQ type 4} and by one $6$-tuple of type \eqref{EQ type 1}. And $g$ can be presented as $g=\sigma^{-4} g'$ for the function
$$
g' =(
0, 0, 0, 0,
3, 9, 3, -6;\;\;
3, 9, 4, -5,
3, 9, 5, -4;\;\;
3, 9, 6, -3,
3, 9, 7, -2;\;\;
3, 9, 8, -1,
3, 9, 9,  0
)
$$ 
which is constructed by one $8$-tuple of type \eqref{EQ type 2} and 
three $8$-tuples of type \eqref{EQ type 4}.
Note that the $3$'rd coordinate in $g$ is $p-q=3-9=-6$.
\end{Example}

After these examples the formal proofs are simpler to understand. Assume a pair $(p,q)$ is chosen, and 
$g$ is any non-zero function in $\mathcal P_2$. Since $g\in \E$, there is a \textit{first} $8$-tuple 
$$
t_i\!=\!\big(g(8i),\,g(8i\!+\!\!1),\,g(8i\!+\!2),g(8i\!+\!3),\;\;\;\,g(8i\!+\!4),\,g(8i\!+\!5),\,g(8i\!+\!6),\,g(8i\!+\!7)\big)
$$
in which $g$ has its first non-zero coordinate.
Using arguments similar to those in
Subsection~\ref{SU Construction of (n,-n)}
we show that if $t_i$ is of type \eqref{EQ type 3}, then a \textit{descending} chain of coordinates $g(8i+2), g(8i+6), g(8i+10), \ldots$ starts from $t_i$. If $p<q$ then this chain never ends, which is a contradiction to the fact that $g\in \E$ has finitely many non-zero coordinates.
If $p\ge q$ then this chain ends either by the $8$-tuple $(p, q, q, 0, \;\;
0, 0, 0, 0)$ of type \eqref{EQ type 1}, or 
by the $8$-tuple $(p, q, q\!+\!1, 1, \;\;
p, q, q, 0)$ of type \eqref{EQ type 3}.
This means that $g(8i\!+\!2)$ is equal to $p-q$.

If $t_i$ is of type \eqref{EQ type 4}, then an \textit{ascending} chain of coordinates starts from $t_i$. If $p>q$, we get a contradiction, and if $p\ge q$, we again get that $g(8i\!+\!2)$ is equal to $p-q$.

Finally, if $t_i$ is of type \eqref{EQ type 2} we get a descending or ascending chain, and then,
$g(8i\!+\!6)$ is equal to $p-q$.

The ``extremal'' case, when $t_i=(p,q,q,0,\;0,0,0,0)$ is of type \eqref{EQ type 1} for a $g$ is possible only if the respective $g'$ either starts by $(0,0,0,0,\;\;p,q,p,n)$ of type \eqref{EQ type 2}\; (i.e., $p=q$, and $n=0$, that is, we again have the equality $p-q=p-p=n=0$), or 
$g'$ starts by $(p,q,m,n,\;p,q,m\!-\!1,n\!-\!1)$ of type \eqref{EQ type 3}\;
(i.e., $p=q=m=n=0$ which leads to a contradiction, as then $t_i$ is a zero $8$-tuple). The case when  $t_i$ is of type \eqref{EQ type 4} is excluded in a similar way.

We see that a sequence $g \in \mathcal P_2$ can consist of a few $8$-tuples (holding a chain of the above types) only. Now we need extract the required fragments $(p,\;q,\;p-q)$.

Clearly $\pi_1 \mathcal P$ consists of sequences of type $h=(p, q, n_2, n_3,\ldots)$ with 
$(p,q) \in \mathcal P$, and with only finitely many of the coordinates 
$n_2, n_3,\ldots$ being non-zero.

Denote $\mathcal P_3 =  \;\iota(\mathcal P_2,\; \pi_1 \mathcal P)$ and choose any $g \in \mathcal P_3$.

Since $g \in \mathcal P_2$, it is constructed by some $8$-tuples of one of the types \eqref{EQ type 1}--\eqref{EQ type 4}. Since also $g \in \pi_1 \mathcal P$, its \textit{first} non-zero $8$-tuple occupies indices $0$--$7$, and is of types 
\eqref{EQ type 2} or \eqref{EQ type 3} with one of $p,q$ being non-zero.
Then by our construction $g$ starts by the triple $(p,\;q,\;p-q)$). 

The case when $\mathcal P$ does contain the couple $(0,0)$ also is covered by our construction because in that case the $8$-tuple of type \eqref{EQ type 1}
with $p=q=0$ is in $\mathcal P_2$, and to 
$\mathcal P_3$ contains a sequence starting by $(0,0,0)$. 

\medskip
The extract set $\epsilon_{0,1,2}\, \mathcal P_3$ is the set of triples $\mathcal P_{1,2,1-2}=\{(p,\;q,\;p-q) \mathrel{| } (p,\;q)\in \mathcal P\}$.

The other set $\mathcal P_{1,2,1+2}$ can now be obtained in two ways. Either we can modify the constructions above to adapt it for the triples $(p,\;q,\;p+q)$. Or we can use the construction of Subsection~\ref{SU Construction of (n,-n)} to build the
set of $\mathcal P_4$ of triples  
$(p,\;q,\;-q)$. Then the extract $\epsilon_{0,2}\;\mathcal P_4$ is the set $\{(p,\;p-q) \mathrel{| } (p,\;q)\in \mathcal P\}$. So we can directly apply the already constructed proof to get the triples $(p,\;-q,\;p-(-q))=(p,\;-q,\;p+q)$, and finally, replace $-q$ by $q$.

We proved the following lemma.

\begin{Lemma} 
\label{LE (p,q,p-q), (p,q,p+q)}
If the sets
$\mathcal P$ belongs to $\BB$,
then the sets 
$\mathcal P_{1,2,1-2}$ and 
$\mathcal P_{1,2,1+2}$
both belong to  $\BB$.
\end{Lemma}

Combining Lemma~\ref{LE (p,q,p-q), (p,q,p+q)}
with Lemma~\ref{LE duplicate k}
and Lemma~\ref{LE (n,-n)}
we get that if $\mathcal Q$ is a set of some $(q)$, then  $\mathcal Q \in \BB$ implies  that $\BB$ contains the set of all couples $(q,q)$,
the set of all triples $(q,q, 2q)$,
and the set of all couples $(q,2q)$ (which is obtained from the set of previous triples via the extract $\epsilon_{0,2}$).
Repeating this we get the set of all couples $(q,s\cdot q)$ for any pre-given integer $s$, and $q\in \mathcal Q$.

\medskip

Lemma~\ref{LE (p,q,p-q), (p,q,p+q)} can also be generalized by taking any distinct indices instead of $0,1,2$. Let $\B$ be any subset of $\E$, and let $p=g(i)$ and $q=g(j)$ be the $i$'th and $j$'th coordinates of generic $g\in \B$. For an index $k$, different from $i,j$ denote by 
$\mathcal P_{i,\,j,\,i-j,\,k}$ the set of all functions $f\in \E$ for which there is a $g\in \B$ such that 
$f$ coincides with $g$ on all coordinates except the $k$'th, and $f(k)=p-q$. In other words, we replace the $k$'th coordinate in each $g\in \B$ by $p-q=g(i)-g(j)$.
We can similarly define the set $\mathcal P_{i,\,j,\,i+j,\,k}$.

\begin{Lemma}
\label{LE (p,q,p+q) at any place}
If the set
$\mathcal B$ belongs to $\BB$,
and for fixed $i,j$ the set $\mathcal P = \big\{(p,q) \mathrel{|} g\in \B,\; p=g(i),\; q=g(j)\big\}$ also belongs to $\BB$,
then the sets 
$\mathcal P_{i,\,j,\,i-j,\,k}$ and 
$\mathcal P_{i,\,j,\,i+j,\,k}$
both belong to  $\BB$.
\end{Lemma}

\begin{proof}
Applying the appropriate permutation $\alpha$ we can reorder the coordinates of each $g\in \B$ so that
$\alpha g (0)=p=g(i)$,\;
$\alpha g (1)=q=g(j)$,\;
$\alpha g (2)=  g(k)$.
Then $\mathcal R_1=\zeta_2\alpha\B$
consists of all those reordered sequences with the $2$'nd coordinate liberated.

Applying Lemma~\ref{LE (p,q,p-q), (p,q,p+q)} to $\mathcal P$ we get that the set 
$\mathcal P_{1,\,2,\,1-2}=\big\{
(p,\,q,\,p-q) \mathrel{|}
(p,\,q)\in \mathcal P 
\big\}$
is in $\BB$.
Then $\mathcal R_2=\pi'\pi_2\mathcal P_{1,\,2,\,1-2}$ consists of all $g\in \E$ which coincide with 
$(p,\,q,\,p-q)$ on indices $0,1,2$, and which may have arbitrary coordinates elsewhere.
Then the intersection $\mathcal R_3=\iota(\mathcal R_1,\mathcal R_2)$
consists of all sequences $g$ from $\mathcal R_1$ in which the $2$'nd coordinate is replaced by 
$p-q=g(0)-g(1)$.
It remains to apply the permutation $\alpha^{-1}$ to get 
$\mathcal P_{i,\,j,\,i-j,\,k}=\alpha^{-1} \mathcal R_3$.

\smallskip
The proof for $P_{i,\,j,\,i+j,\,k}$ is similar.
\end{proof}

\subsection{An application of the $H$-machine}
\label{SU Application of the machine}

Now we are in position to launch the $H$-machine to construct the series of sets $\B$ mentioned in examples in Subsection~\ref{SU The structure of sequences} by Higman operations  \eqref{EQ Higman operations}.
Here we do that for the group $\Z^\infty$.

\begin{Example} 
\label{EX Back to the example for Z^infty}
For the free abelian group $\Z^\infty$ in 
Example~\ref{EX turn back} we have the set $\B$ of sequences $f_{k,l}$ of length $35$ constructed in Example~\ref{EX turn back}. 
The following algorithm constructs $\B$ by operations  \eqref{EQ Higman operations}:

\begin{enumerate}
\item Using the single permutation $\alpha$ constructed in Example~\ref{EX turn back} bring the sequences $f_{k,l}$ to simpler form $\alpha\, f_{k,l}$. 

\item Using Lemma~\ref{LE (n_1), ..., (n_k)} obtain the set $\big\{(1)\big\}$.

\item Using Lemma~\ref{LE duplicate k} ten times, duplicate the $0$'th coordinate $1$ to get the set $\mathcal A_1$ consisting of one  $11$-tuple
$(1,\ldots,1)$.

\item Using Lemma~\ref{LE (n_1), ..., (n_k)} obtain the set $\big\{(-1)\big\}$.

\item Using Lemma~\ref{LE duplicate k} ten times, duplicate the $0$'th coordinate $-1$ to get the set $\mathcal A_2$ consisting of one  $11$-tuple
$(-1,\ldots,-1)$.

\item By Lemma~\ref{LE (n,-n)}
the set $\mathcal B_{+,-}$ of all couples
$(k,-k)$, $k=1,2,\ldots$, 
is in $\BB$. Construct the set $\mathcal A_3$ of all  $5$-tuples 
$(k,k,-k,-k,-k)$, $k=1,2,\ldots$, i.e., duplicate in $\mathcal B_{+\,-}$ the coordinate $-k$ twice by Lemma~\ref{LE duplicate k},
then rotate the resulting set by $\rho$,  duplicate the $0$'th coordinate $k$, and then rotate back by $\rho$, and shift by $\sigma$.

\item Similarly construct the set $\mathcal A_4$ of all  $5$-tuples 
$(l,l,l,-l,-l)$, $l=1,2,\ldots$

\item The sum 
$\mathcal A_5 =\mathcal A_1 
+ \sigma^{11}\mathcal A_2
+ \sigma^{22}\mathcal A_3
+ \sigma^{27}\mathcal A_4
$ consists of $32$-tuples (indexed by $0,1,\ldots,31$):  the above $11$-tuples $(1,\ldots,1)$, followed by $11$-tuples $(-1,\ldots,-1)$, followed by $5$-tuple
$(k,k,-k,-k,-k)$ and then followed by
$5$-tuple $(l,l,l,-l,-l)$ with any 
$k,l=1,2,\ldots$

\item Using Lemma~\ref{LE (p,q,p+q) at any place} on the set $\mathcal A_5$
for 
$i=22$,\; $j=27$,\; $k=32$,
we adjoin a new $32$'rd entry (equal to the respective $k-l$) 
to sequences from $A_5$.
Repeating this step for $k=33$ adjoin a $33$'rd entry $k-l$. Call the new set $\mathcal A_6$.

\item Using Lemma~\ref{LE (p,q,p+q) at any place} 
on $\mathcal A_6$ for 
$i=27$,\; $j=22$,\; $k=34$
we adjoin a new, $34$'th entry $l-k$ to all sequences from $\mathcal A_6$.
I.e., we get the set of all sequences 
$\alpha\, \B$.

\item Apply the inverse $\alpha^{-1}$ of the permutation $\alpha$ used in Example~\ref{EX turn back} to get the set $\B$ of all $f_{k,l}$.

\item As a last step use the definition in Subsection~\ref{SU Extra auxiliary operations} to 
replace by Higman operations  \eqref{EQ Higman operations} each of the auxiliary operations 
$\sigma^i$,\,
$\zeta_i$,\,
$\zeta_{S}$,\,
$\pi_i$,\,
$\pi'_i$,\,
$\tau_{k,l}$,\,
$\alpha=\tau_{k_1,l_1}\cdots \tau_{k_m,l_m}$,\,
$\epsilon_{S}$,\,
addition $+$,\; 
$\iota_n$ that we used in previous steps. 
\end{enumerate}
\end{Example}

\begin{Remark}
\label{RE Comparing the very similarly looking sequences}
Comparing the very similarly structured sequences of  Example~\ref{EX turn back},
Example~\ref{EX embedding of free metabelian into 2-generator group},
Example~\ref{EX embedding of rational group} or  Example~\ref{EX embedding of Pruefer group} 
the reader can see how easy it would be to adapt the above algorithm for free metabelian, soluble, nilpotent groups, for $\Q$ (for \cite{An explicit embedding of Q}), for $\Co_{p^\infty}$,\ or for their direct products including divisible abelian groups, and for any other constructively given subgroups therein. 
\end{Remark}

\vskip-2mm

\medskip
\noindent 
E-mail:
\href{mailto:v.mikaelian@gmail.com}{v.mikaelian@gmail.com}
$\vphantom{b^{b^{b^{b^b}}}}$

\noindent 
Web: 
\href{https://www.researchgate.net/profile/Vahagn-Mikaelian}{researchgate.net/profile/Vahagn-Mikaelian}

\end{document}